\documentclass{amsart}

\usepackage[utf8]{inputenc}

\usepackage{amssymb}

\usepackage{color}

\usepackage{hyperref}

\usepackage{amsrefs}

\newtheorem{thm}{Theorem}[section]
\newtheorem{prop}[thm]{Proposition}
\newtheorem{lem}[thm]{Lemma}
\newtheorem{cor}[thm]{Corollary}
\newtheorem{con}[thm]{Conjecture}

\newtheorem{thmalpha}{Theorem}

\newtheorem{lemalpha}[thmalpha]{Lemma}

\theoremstyle{definition}

\theoremstyle{remark}

\newtheorem{remark}[thm]{Remark}

\numberwithin{equation}{section}

\newcommand{\Rea}{\operatorname{Re}} 

\newcommand{\C}{\mathbb{C}} 
 
\newcommand{\D}{\mathbb{D}}

\newcommand{\B}{\mathcal{B}}
\newcommand{\Hol}{\mathcal{H} \left( \mathbb{D} \right)}

\begin{document}


	\title{Norms of inclusions between some spaces of analytic functions}


	\author{Adrián Llinares}
	\address{Departamento de Matemáticas \\ Universidad Autónoma de Madrid \\ 28049 Madrid, Spain}
	\email{adrian.llinares@uam.es}
	
	\date{\today}


	\begin{abstract}
		The inclusions between the Besov spaces $B^q$, the Bloch space $\mathcal{B}$ and the standard weighted Bergman spaces $A^p_\alpha$ are completely understood, but the norms of the corresponding inclusion operators are in general unknown. In this work, we compute or estimate asymptotically the norms of these inclusions.
	\end{abstract}

	\maketitle



	\section{Introduction}
		
		As is usual, denote $\D$ the open unit disk on the complex plane $\C$. Let $\Hol$ the set of all analytic functions on $\D$. Given $0 < p < \infty$, $r \in [0, 1)$ and $f\in \Hol$, let $M_p(r,f)$ be the integral $p$-mean of $|f|$ over the circle of radius $r$,
			\[
				M_p (r, f) := \left(\dfrac{1}{2\pi} \int_0^{2\pi} |f(re^{it})|^p \: dt \right)^\frac{1}{p}.
			\]
		
		Consider $dA(z) = \frac{1}{\pi} \: dx \, dy = \frac{r}{\pi} \: dr \, dt$, the normalized area measure in $\D$. For any $\alpha > -1$, we define the standard weighted Bergman space $A^p_\alpha$ as the set of all holomorfic functions $f$ such that
			\[
				\| f \|_{A^p_\alpha} := \left( \int_\D |f|^p \: d\mu_\alpha \right)^\frac{1}{p} = \left( (\alpha + 1) \int_0^1 2r (1 - r^2)^\alpha M_p^p (r,f) \: dr \right)^\frac{1}{p} < \infty,
			\]
			where for sake of simplicity $d\mu_\alpha (z) := (\alpha + 1) (1 - |z|^2)^\alpha \: dA(z)$. If $\alpha = 0$, we will simply use the notation $A^p$ instead of $A^p_0$.
		
		If $q > 1$ we say that $f \in \Hol$ belongs to the (analytic) Besov space $B^q$ whenever its derivative $f'$ is included in $A^q_{q - 2}$. $B^q$ is a Banach space when equipped with any of the following two norms:
			\begin{eqnarray*}
				\| f \|_{B^q,1} & := & |f(0)| + \| f' \|_{A^q_{q - 2}}, \\
				\| f \|_{B^q,2} & := & \left( |f(0)|^q + \| f' \|_{A^q_{q - 2}}^q \right)^\frac{1}{q}.
			\end{eqnarray*}
		
		Clearly the two norms are equivalent, and the following sharp inequalities hold
			\[
				\| f \|_{B^q, 2} \leq \| f \|_{B^q, 1} \leq 2^\frac{q - 1}{q} \| f \|_{B^q, 2}, \quad \forall f \in B^q.
			\]
		
		If $q = 1$, this definition is not appropriate since the space would only contain the constant functions, so a different definition is needed. Thus, the space $B^1$ is usually defined as the set of all functions $f$ such that
			\begin{equation} \label{DefB1}
				f(z) = \sum_{k = 1}^\infty b_k \varphi_{a_k} (z),
			\end{equation}
			for some $\{a_k \}_{k \geq 1} \subset \D$ and $\{ b_k \}_{k \geq 1} \in \ell^1$. Here $\varphi_a$ is used in order to denote the disk automorphism
			\[
				\varphi_a (z) := \dfrac{a - z}{1 - \overline{a} z}.
			\]
			
		If $\| f \|_{B^1} := \inf \{ \| \{ b_k \}_{k \geq 1} \|_{\ell^1} \}$, it can be shown \cite{MR814017} that there exist two absolute constants $A, B > 0$ such that
			\begin{equation} \label{AlternativeB1}
				A \| f'' \|_{A^1} \leq \| f - f(0) - f'(0) z \|_{B^1} \leq B \| f'' \|_{A^1}, \quad \forall f \in B^1,
			\end{equation}
			and consequently an alternative characterization of $B^1$ can be given in terms of the integrability of the second derivative, similarly to the way of defining $B^q$ for $q > 1$. Here and also elsewhere in the paper, we will also write \eqref{AlternativeB1} as
			\[
				\| f - f(0) - f'(0)z \|_{B^1} \approx \| f'' \|_{A^1}.
			\]
			If only one inequality holds, we will use the symbol $\lesssim$ or $\gtrsim$ instead.
			
		Finally, the Bloch space $\B$ is the  class of all the analytic functions such that
			\[
				\rho_\B(f) := \sup_{z \in \D} \{ |f'(z)|(1 - |z|^2) \} < \infty.
			\]
			The expression $\rho_\B$ is actualy a (complete) seminorm, and thus $\B$ is a Banach space with respect to the norm
			\[
				\| f \|_\B := |f(0)| + \rho_\B (f).
			\]
		
		The Bloch space is not separable, and therefore the closure of the polynomials is a proper subspace of $\B$. This subspace is called the little Bloch space $\B_0$. It can be proven \cite{MR0361090} that
			\[
				f \in \B_0 \quad \Longleftrightarrow \quad \lim_{|z| \rightarrow 1^{-}} \{ |f'(z)| (1 - |z|^2) \} = 0.
			\]
		
		As a last word regarding the notation, we are going to write
			\[
				B(x,y) := \int_0^1 t^{x-1} (1-t)^{y-1} \: dt \quad \mbox{ and } \quad  \Gamma(x) := \int_0^\infty t^{x-1} e^{-t} \: dt, \quad x, y > 0
			\]
			to denote the Euler Beta and Gamma functions, respectively.
			
		It is well-known that the following chain of inclusions holds
			\[
				B^q \subset \B \subset A^p_\alpha, \quad \mbox{ if } q \geq 1\mbox{, } p > 0 \mbox{ and } \alpha > -1,
			\]
			hence, as a consequence of Closed Graph Theorem ($A^p_\alpha$ is not a normed space if $p < 1$, but it is complete and invariant under translations also in this case) the inclusion operators between them are bounded. The norm of these inclusions are in general either not known or difficult to find in the literature.
			
		We are interested in estimating as precisely as possible the norms of these two inclusions (the third inclusion $B^q \subset A^p_\alpha$ can be trivially estimated by composition).
		
		More specifically, the main results of this work can be listed as follows:
			\begin{itemize}
				\item $\| f \|_\B \leq \| f \|_{B^q, 1}$ and $\| f \|_\B \leq 2^\frac{q-1}{q} \| f \|_{B^q,2}$.
				
				\item $\| f \|_\B \leq 2 \| f \|_{B^1}$.
				
				\item If $p \geq 1$, $\sup \left \{ \| f \|_{A^p_\alpha} \: : \: \| f \|_\B \leq 1 \right \} \approx p$.
				
				\item If $\alpha \geq 0$, there exists $p_\alpha > 2 \max \left \{ 1, B^{-1} \left ( \frac{1}{2}, \alpha + 1 \right) \right \}$ such that $\| f \|_{A^p_\alpha} \leq \| f \|_\B$ for every $p \in (0, p_\alpha]$.
			\end{itemize}
		
	\section{The inclusion of $B^q$ in $\B$}
		
		There are many ways of proving that $B^q \subset \B$. Maybe the most elegant one is to show that $B^q$ is conformally invariant and using the fact that $\B$ is (under some conditions) maximal in this class of spaces \cite{MR529210}, as a counterpart of $B^1$, which is minimal \cite{MR814017}.
		
		However, it is easier to prove that $B^q \subset \B$ via the pointwise bounds of $A^p_\alpha$ functions (see \cite{MR1120512} or Lemma 3.2 in \cite{MR1758653}, for example).
		
		\begin{lemalpha}[Pointwise estimates in $A^p_\alpha$]
			If $p > 0$, $\alpha > -1$ and $\zeta \in \D$, then
				\[
					|g(\zeta)| \leq \dfrac{ \| g \|_{A^p_\alpha}}{(1 - |\zeta|^2)^\frac{\alpha + 2}{p}}, \quad \forall g \in A^p_\alpha.
				\]
				
			Moreover, equality holds if and only if $g$ is a multiple of
				\[
					k_\zeta (z) := \dfrac{1}{(1 - \overline{\zeta} z)^\frac{2(\alpha + 2)}{p}}, \quad \forall z \in \D.
				\]
		\end{lemalpha}
		
		If $f \in B^q$ for some $q > 1$, then
			\[
				|f'(z)| (1 - |z|^2) \leq \| f' \|_{A^q_{q-2}}, \quad \forall z \in \D,
			\]
			and therefore $f \in \B$.
			
		On the other hand, if $f \in B^1$ then $f$ is indeed a bounded analytic function and due to Schwarz-Pick Lemma (after a suitable normalization) we get that $f \in \B$.
		
		At this point, we are able to compute the norm of the inclusion $B^q \subset \B$.
		
		\begin{prop}
			\
			\begin{enumerate}
				\item Let $q > 1$. Then
						\[
							\| f \|_\B \leq \| f \|_{B^q, 1} \quad \mbox{ and } \quad \| f \|_\B \leq 2^\frac{q - 1}{q} \| f \|_{B^q, 2}, \quad \forall f \in B^q.
						\]
					
					Moreover, if the equality is attained in any of the above inequalities, then $f$ must be of either of the following forms
						\begin{eqnarray*}
							f_0 (z) & := & \alpha z + \beta, \\
							f_\zeta (z) & := & \gamma \dfrac{1 - |\zeta |^2}{\overline{\zeta}} \dfrac{1}{1 - \overline{\zeta} z} + \delta, \quad 0 < |\zeta | < 1,
						\end{eqnarray*}
						for some suitable constants $\alpha$, $\beta$, $\gamma$ and $\delta$.
				
				\item The following inequality is sharp
						\[
							\| f \|_\B \leq 2 \| f \|_{B^1}, \quad \forall f \in B^1.
						\]
						
						Moreover, equality is attained only if $f \equiv 0$.
			\end{enumerate}
		\end{prop}
		
		\begin{proof}
			\
			\begin{enumerate}
				\item If $f \in B^q$, clearly
						\[
							\| f \|_\B = |f(0)| + \rho_\B (f) \leq  |f(0)| + \| f' \|_{A^q_{q - 2}} = \| f \|_{B^q,1} \leq 2^\frac{q - 1}{q} \| f \|_{B^q, 2}.
						\]
					
					Observe that if either $\| f \|_\B = \| f \|_{B^p,1}$ or $\| f \|_\B = 2^\frac{q-1}{q} \| f \|_{B^q, 2}$, then necessarily $\rho_\B (f) = \| f' \|_{A^q_{q-2}}$. Note that the polynomials are dense in $B^q$ and the inclusion operator is continuous, so indeed $B^q \subset \B_0$. Thus, there exists $\zeta \in \D$ such that
						\[
							|f'(\zeta)| (1 - |\zeta|^2) = \rho_\B (f) = \| f' \|_{A^q_{q-2}},
						\]
						and therefore we have that $f'(z) = c k_\zeta (z) = c (1 - \overline{\zeta} z)^{-2}$ for some constant $c$, from where we deduce the expression of $f$.
				
				\item Let $f \in B^1$ such that $\| f \|_{B^1} = 1$. As is usual, let $\| f \|_{\infty} := \sup \left \{ |f(z)| : z \in \D \right \}$. If $\{b_k\}_{k \geq 1}$ is an admissible sequence of coefficients for $f$ in the formula \eqref{DefB1}, then it is clear that $\| f \|_\infty \leq \| \{b_k\}_{k \geq 1} \|_{\ell^1}$, so $\| f \|_\infty \leq \| f \|_{B^1}$. Thus, we can apply the Schwarz-Pick Lemma to the function $\frac{f}{\| f \|_\infty}$ in order to get $|f'(z)|(1 - |z|^2) \leq \| f \|_\infty$ for every $z$ in $\D$.
						
					Summing up,
						\[
							\| f \|_\B = |f(0)| + \rho_\B(f) \leq 2 \| f \|_\infty \leq 2.
						\]
					
					This inequality is sharp, because if we test with the sequence of automorphisms $\left \{ \varphi_{1 - k^{-1}} \right \}_{k \geq 1}$, it is immediate to check that $\| \varphi_{1 - k^{-1}} \|_{B^1} = 1$ and
						\[
							\left \| \varphi_{1 - \frac{1}{k}} \right \|_\B = 2 - \dfrac{1}{k}, \quad \forall k \geq 1.
						\]
					
					However, if equality holds for a non-identically zero function $f$ (whose $B^1$ norm is assumed to be 1), then in particular $\| f \|_\B = 2 \| f \|_\infty = 2$. This leads to
						\[
							|f(0)| = \rho_\B(f) = \| f \|_\infty = 1,
						\]
						which is in a clear contradiction with the Open Mapping Theorem. Therefore, there are no extremal functions for this inclusion.
			\end{enumerate}
		\end{proof}
		
		\begin{remark}
			\
			\begin{enumerate}
				\item It can be checked that $\| f_\zeta \|_\B = \| f_\zeta \|_{B^q,1}$, $\zeta \in \D$, for any choice of $\alpha$, $\beta$, $\gamma$ and $\delta$.
				
					On the other hand $\| f_\zeta \|_\B = 2^\frac{q-1}{q} \| f_\zeta \|_{B^q, 2}$, whenever $|\alpha| = |\beta| = 2^{-\frac{1}{q}} \| f_0 \|_{B^q, 2}$, if $\zeta = 0$, or $|\gamma| = \left | \gamma \frac{1 - |\zeta|^2}{\overline{\zeta}} + \delta \right | = 2^{-\frac{1}{q}} \| f_\zeta \|_{B^q, 2}$ otherwise.
			
				\item The inequality $\| f \|_\B \leq 2 \| f \|_\infty$ is widely known and can be found in many references (like the remarkable work of Anderson, Clunie and Pommerenke \cite{MR0361090}), but the non-existence of extremal functions, although rather straightforward, is not explicitly stated in the literature.
			\end{enumerate}
		\end{remark}
		
	\section{The inclusion of $\B$ in $A^p_\alpha$}
	
		As a consequence of the definition of the seminorm,
			\begin{equation} \label{PBoundsBloch}
				|f(z) - f(0)| \leq \dfrac{1}{2} \log \left( \dfrac{1 + |z|}{1 - |z|} \right) \rho_\B (f), \quad \forall z \in \D.
			\end{equation}
		
			
		This growth at most logarithmic yields that $\B \subset A^p_\alpha$ for every $p > 0$ and $\alpha > -1$. Moreover, it is sufficient in order to prove the compactness of the inclusion operator.
		
		\begin{lem}
			The inclusion of $\B$ in $A^p_\alpha$ is a compact operator.
		\end{lem}
		
		\begin{proof}
			Let $\{ f_n \}_{n \geq 1}$ a sequence in $\B$ such that $\| f_n \|_\B \leq 1$. As a consequence of \eqref{PBoundsBloch}, $\{ f_n \}_{n \geq 1}$ is uniformly bounded on compact sets of $\D$, and therefore there exists a subsequence $\{ f_{n_k} \}_{k \geq 1}$ uniformly convergent on compact sets to some analytic function $f$.
			
			Note that
				\[
					|f_{n_k} (z)|^p \leq \left [ 1 + \dfrac{1}{2} \log \left (\dfrac{1 + |z|}{1 - |z|} \right) \right ]^p, \quad \forall z \in \D,
				\]
				and the right-hand side of this last inequality is integrable with respect the measure $\mu_\alpha$, hence due to Dominated Convergence Theorem it follows that $\| f_{n_k} \|_{A^p_\alpha}$ converges to $\| f \|_{A^p_\alpha}$. This convergence of the $A^p_\alpha$ norms and the convergence $\mu_\alpha$-a.e. are enough, using the Riesz's Lemma (see \cite{MR0268655}, Lemma 1 in Chapter 2), to prove that $\displaystyle{\lim_{k \rightarrow \infty} f_{n_k} = f}$ in the $A^p_\alpha$ topology.
		\end{proof}
		
		For the sake of simplicity, denote
			\begin{equation} \label{NormBlochAp}
				C_\alpha (p) := \max \left \{ \| f \|_{A^p_\alpha} \: : \: \| f \|_\B \leq 1 \right \}.
			\end{equation}
			
		Note that the maximum in \eqref{NormBlochAp} is attained because the inclusion is a compact operator.
		
		Given $f \in \Hol$, due to an elementary fact from Measure Theory
			\[
				\lim_{p \rightarrow \infty} \| f \|_{A^p_\alpha} = \| f \|_\infty,				
			\]
			so the value $C_\alpha (p)$ must blow up whenever $p$ tends to infinity.
			
		First, we are going to deduce the order of growth of $C_\alpha(p)$. The following lemma will be needed.
		
		\begin{lemalpha}[Chebyshev's inequality]
			Let $\sigma$ a probability measure on a real interval $(a, b)$. If $f$ and $g$ are two increasing functions such that $f$, $g$ and $fg$ are integrable with respect $\sigma$, then
				\[
					\int_a^b f g \: d\sigma \geq \int_a^b f \: d\sigma \int_a^b g \: d\sigma.
				\]
		\end{lemalpha}
		
		\begin{proof}
			The proof is a direct consequence of the fact that the function $\big ( f(x) - f(y) \big) \big( g(x) - g(y) \big)$ is non-negative and Fubini's Theorem.
		\end{proof}
		
		\begin{thm} \label{GrowthOrder}
			Let $\alpha > -1$ fixed. If $p \geq 1$, then $C_\alpha (p) \approx p$.
		\end{thm}
		
		\begin{proof}
			Take $f \in \B$, and assume that $f(0) = 0$.
			
			Consider the probability measure $d\sigma (r) := \frac{2(1-r^2)^\alpha \: dr}{B \left(\frac{1}{2}, \alpha + 1\right)}$. Note that
				\[
					\int_0^1 2r \: d\sigma(r) = \frac{2}{(\alpha + 1) B \left( \frac{1}{2}, \alpha + 1 \right)}.
				\]
						
			Due to the inequality (see for example \cite{MR0268655}, page 82)
				\begin{equation} \label{EstimateDerMp}
					\dfrac{d}{dr} M_p^p (r,f) \leq p M_p^{p-1}(r,f) M_p(r,f'), \quad \forall r \in (0, 1),
				\end{equation}
				integration by parts and Chebyshev's inequality, it follows that
				\begin{eqnarray*}
					\| f \|_{A^p_\alpha}^p & = & \int_0^1 (1 - r^2)^{\alpha + 1} \dfrac{d}{dr} M_p^p (r,f) \: dr \leq \dfrac{B \left( \frac{1}{2}, \alpha + 1\right)}{2} p \int_0^1 M_p^{p-1} (r, f) \: d\sigma(r) \| f \|_\B \\
						& = & \left (\dfrac{B \left( \frac{1}{2}, \alpha + 1\right)}{2} \right)^2 p (\alpha + 1) \int_0^1 2r \: d\sigma(r) \int_0^1 M_p^{p-1} (r, f) \: d\sigma(r) \| f \|_\B \\
						& \leq & \dfrac{B \left( \frac{1}{2}, \alpha + 1\right)}{2} p (\alpha + 1) \int_0^1 2r (1 - r^2)^\alpha M_p^{p-1} (r, f) \: dr \| f \|_\B,
				\end{eqnarray*}
				and the estimate $\| f \|_{A^p_\alpha} \leq \frac{B \left(\frac{1}{2}, \alpha + 1 \right)}{2} p \| f \|_\B$, if $f(0) = 0$, follows from Hölder's inequality.
			
			The Bloch seminorm $\rho_B$ is invariant under translations. Thus, for every $f \in \B$
				\[
					\| f \|_{A^p_\alpha} \leq |f(0)| + \| f - f(0) \|_{A^p_\alpha} \leq \max \left \{ \dfrac{B \left( \frac{1}{2}, \alpha + 1\right)}{2}, 1 \right \} p \| f \|_\B.
				\]
			
			That is, $C_\alpha (p) \lesssim p$.
			
			In order to get the reverse inequality, we are going to test with the function $f(z) = -\frac{1}{2} \log (1 - z)$. It is immediate to check that $\| f \|_\B = 1$ and using polar coordinates centered at 1
				\begin{eqnarray*}
					\| f \|_{A^p_\alpha}^p & \geq & \dfrac{\alpha + 1}{\pi 2^p} \int_{\frac{\pi}{2}}^{\frac{3\pi}{2}}  \int_0^{2|\cos t|} r^{\alpha + 1} (2|\cos t| - r)^\alpha \left | \log \dfrac{1}{r} \right|^p \: dr \, dt \\
						& \geq & \dfrac{\alpha + 1}{\pi 2^{p-1}} \int_{\frac{2\pi}{3}}^\pi \int_0^1 r^{\alpha + 1} (2|\cos t| - r)^\alpha \log^p \dfrac{1}{r} \: dr \, dt.
				\end{eqnarray*}
			
			Observe that the factor $(2|\cos t| - r)^\alpha$ is a monotonic function with respect $r$, so if we write
				\[
					M_\alpha := \dfrac{1}{\pi} \int_{\frac{2\pi}{3}}^\pi \min \left \{ (2|\cos t|)^\alpha, (2|\cos t| - 1)^\alpha  \right \} \: dt,
				\]
				then
				\[
					C_\alpha^p (p) \geq \| f \|_{A^p_\alpha}^p \geq \dfrac{M_\alpha}{2^{p-1}} \int_0^1 r^{\alpha + 1} \log^p \dfrac{1}{r} \: dr = \dfrac{M_\alpha}{2^{p-1}(\alpha + 2)^{p+1}} \Gamma (p + 1).
				\]
				
			Consequently, the inequality $C_\alpha(p) \gtrsim p$ holds due to Stirling's formula.
		\end{proof}
		
			Indeed, the proof above can be slightly modified to demonstrate the contractivity of the inclusion for a wide range of exponents (depending on $\alpha$, of course).
			
		\begin{thm} \label{Contractive1}
			Let $\alpha \geq 0$. If $p \leq \frac{2}{B \left( \frac{1}{2}, \alpha + 1 \right)}$, then
				\[
					\| f \|_{A^p_\alpha} \leq \| f \|_\B, \quad \forall f \in \B.
				\]
			
			In addition, the equality is attained if and only if $f$ is constant.
			
			In particular, $C_\alpha (p) = 1$ for all these values of $p$.
		\end{thm}
		
		\begin{proof}
			Note that $\mu_\alpha$ is a probability measure, so it is enough to prove the contractivity for $ p = \frac{2}{B \left( \frac{1}{2}, \alpha + 1 \right)}$. This quantity is strictly increasing in $\alpha$, and therefore $p \geq 1$ because of $\alpha \geq 0$. Then,
				\[
					\| f \|_{A^p_\alpha} \leq |f(0)| + \| f - f(0) \|_{A^p_\alpha} \leq |f(0)| + \dfrac{B \left(  \frac{1}{2}, \alpha + 1 \right)}{2} p \rho_\B (f) = \| f \|_\B.
				\]
				
			If equality holds then the Bloch seminorm is attained at every point. That is,
				\[
					|f'(z)| = \dfrac{\rho_\B(f)}{1 - |z|^2}, \quad \forall z \in \D,
				\]
				but an elementary application of the Maximum Modulus Principle and the Identity Principle shows that the only analytic functions whose moduli are radial functions are the monomials, hence necessarily $\rho_\B (f) = 0$ and $f$ is constant.
		\end{proof}
		
		However, this range of contractivity is not, in general, sharp. We are going to prove that, por any $\alpha \geq 0$ the inclusion of $\B$ in the Hilbert space $A^2_\alpha$ is also contractive (a fact that it is not covered in the last Theorem for $\alpha$ small enough).
		
		A sketch of the proof for the unweighted case $\alpha = 0$ can be found in \cite{MR0361090}, where the expression of $\| f \|_{A^2}$ in terms of the Taylor coefficients of $f$ is needed. In the general case, our main idea relies upon the so called Hardy-Stein identity (\cite{MR1217706}, Section 8.2), or simply Hardy identity in some references,
			\[
				\dfrac{d}{dr} M_p^p (r,f) = \dfrac{p^2}{2r} \int_{r\D} |f'(z)|^2 |f(z)|^{p-2} \: dA(z), \quad \forall r \in (0,1),
			\]
			which can be understood as a refinement of \eqref{EstimateDerMp}.
		
		\begin{thm} \label{Contractive2}
			If $\alpha \geq 0$, then
				\[
					\| f \|_{A^2_\alpha} \leq \| f \|_\B, \quad \forall f \in \B,
				\]
				and equality is attained only for constant functions.
		\end{thm}
		
		\begin{proof}
			Take $f \in \B$. Using integration by parts and the Hardy-Stein identity, we obtain
				\begin{eqnarray*}
					\| f \|_{A^2_\alpha}^2 & = & |f(0)|^2 + 2 \int_0^1 \dfrac{(1 - r^2)^{\alpha + 1}}{r} \int_{r\D} |f'(z)|^2 \: dA(z) \, dr \\
						& \leq & |f(0)|^2 + \dfrac{\rho_\B(f)^2}{\alpha + 1} \leq \| f \|_\B^2.
				\end{eqnarray*}
				
			Again, $\| f \|_{A^2_\alpha} = \| f \|_\B$ implies that $\rho_\B(f)$ is attained at every point, and thus $f$ is constant.
		\end{proof}
		
		At this point, we realize that the functions that vanish at the origin play an important role, and consequently we are going to write $\tilde{C}_\alpha(p)$ to denote the extremal problem
			\[
				\tilde{C}_\alpha(p) = \max \left \{ \| f \|_{A^p_\alpha} \: : \: \| f \|_\B \leq 1, f(0) = 0 \right \}.
			\]
		
		As an initial remark, note that Theorem \ref{GrowthOrder} holds also for $\tilde{C}_\alpha (p)$. That is, $\tilde{C}_\alpha (p) \approx p$ as well.
		
		It is trivial that $\tilde{C}_\alpha (p) \leq C_\alpha(p)$ for every $p$, but if both $p$ and $\tilde{C}_\alpha(p)$ are greater than or equal to 1, due to the triangle inequality
			\[
				\| f \|_{A^p_\alpha} \leq |f(0)| + \tilde{C}_\alpha (p) \rho_\B (f) \leq \tilde{C}_\alpha (p) \| f \|_\B, \quad \forall f \in \B,
			\]
			and therefore $C_\alpha (p) = \tilde{C}_\alpha (p)$ in this case.
		
		Similarly, $\tilde{C}_\alpha (p) < 1$ implies that $C_\alpha (p) = 1$. In other words,
			\begin{equation} \label{RemarkBAp}
				C_\alpha (p) = \max \left \{ 1, \tilde{C}_\alpha(p) \right \}, \quad \forall p \geq 1.
			\end{equation}
			
		Observe that $\tilde{C}_\alpha$ is strictly increasing due to the existence of extremal functions and Hölder's inequality.
		
		
		\begin{cor} If $\alpha \geq 0$, there exists $p_\alpha > 2 \max \left \{ 1, B^{-1} \left( \frac{1}{2}, \alpha + 1 \right)\right \}$ such that
			\[
				\| f \|_{A^p_\alpha} \leq \| f \|_\B, \quad \forall f \in \B,
			\]
			for any $p \in (0, p_\alpha]$.
		\end{cor}
		
		\begin{proof}
			It is a consequence of \eqref{RemarkBAp} and Theorems \ref{Contractive1} and \ref{Contractive2}.
		\end{proof}
		
		However, this Corollary does not imply that this improvement is ``significantly'' large. In the unweighted case, the model function $f(z) = \frac{1}{2} \log \frac{1 + z}{1 - z}$ verifies, after some suitable changes of variables, that
			\[
				\| f \|_{A^{\frac{25}{4}}}^\frac{25}{4} = \dfrac{1}{\pi 2^\frac{9}{4}} \int_0^\infty \int_0^{\frac{\pi}{2}} \dfrac{(x^2 + y^2)^\frac{25}{8}}{\big( e^x + 2\cos(y) + e^{-x} \big)^2} \: dy \, dx > 1,
			\]
			and therefore $p_0 \in \left(2, \frac{25}{4} \right)$. This last inequality has been checked by a numerical computation using the software \textit{Wolfram Alpha}.
		
		Next, since $C_\alpha$ and $\tilde{C}_\alpha$ will be eventually equal, we are going to exploit the restriction of $\tilde{C}_\alpha$ in order to guess the correct asymptotics for $C_\alpha$.
		
		\begin{lem} \label{LemParseval}
			If $\alpha > -1$ and $f(z) = \displaystyle{\sum_{n = 0}^\infty a_n z^n}$ is analytic in $\D$, then
				\[
					(\alpha + 1) (\alpha + 2) \sum_{n = 1}^\infty \dfrac{n}{n + \alpha + 2} \dfrac{\Gamma(\alpha + 2) n!}{\Gamma(n + \alpha + 2)} |a_n|^2 = \int_\D |f'(z)|^2 (1 - |z|^2)^2 \: d\mu_\alpha (z).
				\]
				
			In addition, if $f(0) = \ldots = f^{(k-1)}(0) = 0$ for some $k \geq 1$,
				\[
					\| f \|_{A^2_\alpha}^2 \leq \dfrac{k + \alpha + 2}{(\alpha + 1)(\alpha + 2)k} \int_\D |f'(z)|^2 (1 - |z|^2)^2 \: d\mu_\alpha (z).
				\]
		\end{lem}
		
		\begin{proof}
			This is a consequence of Parseval's identity and the monotonicity of the sequence $\left \{ \frac{n}{n + \alpha + 2} \right \}_{n \geq 1}$.
		\end{proof}
		
		\begin{thm} \label{Bound2n}
			For any $\alpha > -1$ and $n \geq 2$
				\[
					\tilde{C}_\alpha (2n) \leq \dfrac{1}{\sqrt{(\alpha + 1)(\alpha + 2)}} \left( \dfrac{(\alpha + 1)(\alpha + 2) \Gamma(n + \alpha + 3)n!}{\Gamma(\alpha + 4)} \tilde{C}_\alpha^2 (2) \right)^\frac{1}{2n}.
				\]
		\end{thm}
		
		\begin{proof}
			We give a proof by Induction. Let $f \in \B$ with $f(0) = 0$ and $\| f \|_\B = 1$. Note that $f^2$ verifies that $f^2 (0) = (f^2)'(0) = 0$ and due to Lemma \ref{LemParseval}
				\begin{eqnarray*}
					\| f \|_{A^4_\alpha}^4 = \| f^2 \|_{A^2_\alpha}^2 & \leq & \dfrac{\alpha + 4}{2(\alpha + 1)(\alpha + 2)} \int_\D |(f^2)'(z)|^2 (1 - |z|^2)^2 \: d\mu_\alpha (z) \\
						& \leq & \dfrac{2(\alpha + 4)}{(\alpha + 1)(\alpha + 2)} \| f \|_{A^2_\alpha}^2 \leq \dfrac{2(\alpha + 4)}{(\alpha + 1)(\alpha + 2)} \tilde{C}_\alpha^2 (2),
				\end{eqnarray*}
				and hence the statement has been proved for $n = 2$.
				
			Now assume
				\[
					\tilde{C}_\alpha (2n) \leq \dfrac{1}{\sqrt{(\alpha + 1)(\alpha + 2)}} \left( \dfrac{(\alpha + 1)(\alpha + 2) \Gamma(n + \alpha +3) n!}{\Gamma(\alpha + 4)} \tilde{C}_\alpha^2 (2) \right)^\frac{1}{2n},
				\]
				for some $n \geq 2$.
				
			Again, it is immediate that the first $n + 1$ Taylor coefficients of $f^{n + 1}$ are zero and therefore
				\begin{eqnarray*}
					\| f \|_{A^{2(n+1)}_\alpha}^{2(n+1)} = \| f^{n+1} \|_{A^2_\alpha}^2 & \leq & \dfrac{n + \alpha + 3}{(\alpha + 1)(\alpha + 2)(n + 1)} \int_\D |f^{(n+1)}(z)|^2 (1 - |z|^2)^2 \: d\mu_\alpha (z) \\
						& \leq & \dfrac{(n + \alpha + 3)(n + 1)}{(\alpha + 1)(\alpha + 2)} \| f \|_{A^{2n}}^{2n} \\
						& \leq & \dfrac{\Gamma(n + \alpha + 4) (n + 1)!}{(\alpha + 1)^n (\alpha + 2)^n \Gamma(\alpha + 4)} \tilde{C}_\alpha^2(2),
				\end{eqnarray*}
				which proves the Theorem.
		\end{proof}
		
		\begin{cor} \label{AsympBloch}
			If $\alpha > -1$, then
				\begin{eqnarray*}
					\liminf_{p \rightarrow \infty} \dfrac{C_\alpha (p)}{p} & \geq & \dfrac{1}{2e(\alpha + 2)}, \\
					\limsup_{p \rightarrow \infty} \dfrac{C_\alpha (p)}{p} & \leq & \dfrac{1}{2e \sqrt{(\alpha + 1)(\alpha + 2)}}.
				\end{eqnarray*}
		\end{cor}
		
		\begin{proof}
			The lower bound for the limit inferior can be deduced from the proof of Theorem \ref{GrowthOrder} and, of course, Stirling's formula for $\Gamma$.
			
			On the other hand, if $p$ is large enough
				\[
					\dfrac{C_\alpha (p)}{p} = \dfrac{\tilde{C}_\alpha (p)}{p} \leq \dfrac{2 \left \lceil \frac{p}{2} \right \rceil}{p} \dfrac{\tilde{C}_\alpha \left( 2 \left \lceil \frac{p}{2} \right \rceil \right)}{2 \left \lceil \frac{p}{2} \right \rceil},
				\]
				where $\lceil \cdot \rceil$ is the greatest integer part, so we can use Theorem \ref{Bound2n} to get the upper bound for the limit superior.
		\end{proof}
		
			Observe that both bounds are asymptotically equivalent if $\alpha$ tends to infinity, so	there is the strong feeling that the quotient $\frac{C_\alpha(p)}{p}$ might be convergent with respect $p$, not only bounded.
			
			Moreover, it is widely know that $\B$ is not contained in any Hardy space $H^p$ (that is, the set of all analytic functions whose $M_p$ means are uniformly bounded), which can be understood as the limit space $A^p_{-1}$ in the sense that
				\[
					\| f \|_{H^p} = \lim_{\alpha \rightarrow -1^+} \| f \|_{A^p_\alpha}, \quad \forall f \in H^p.
				\]
				
				Thus, the accurate asymptotic approximation should blow up whenever $\alpha$ tends to $-1$. Given these considerations, we conjecture that $\frac{C_\alpha (p)}{p}$ is actually convergent for any $\alpha > -1$ and the limit should coincide with the bound for the limit superior given in Corollary \ref{AsympBloch}.
			
		\begin{con}
			For any $\alpha > -1$,
				\[
					\lim_{p \rightarrow \infty} \dfrac{C_\alpha(p)}{p/\big(2e \sqrt{(\alpha + 1)(\alpha + 2)} \big)} = 1.
				\]
		\end{con}
		
		We end the paper by deducing an integral condition that the extremal functions must verify.
		
		\begin{thm}
			Let $\alpha > - 1$, $p > 1$ and $f \in \B$ non-identically zero such that $f(0) = 0$ and $\| f \|_{A^p_\alpha} = \tilde{C}_\alpha (p) \| f \|_\B$. Then $f$ is a solution of the functional equation
				\[
					\int_\D |f(z)|^p z \: d\mu_\alpha (z) = \dfrac{p}{2(\alpha + 2)} \overline{f'(0)} \int_\D |f(z)|^{p - 2} f(z) \: d\mu_\alpha (z).
				\]
		\end{thm}
		
		\begin{proof}
			Take $f \in \B$, $f \not \equiv 0$, such that $f(0) = 0$ and $\| f \|_{A^p_\alpha} = \tilde{C}_\alpha (p) \| f \|_\B$.
			
			Consider the family of functions
				\[
					g_a (z) := f \big ( \varphi_a (z) \big) - f(a), \quad a \in \D.
				\]
			
			It is clear that $g_a(0) = 0$ for every $a$, and the Bloch space is strictly conformally invariant (\emph{i.e.}, $\rho_\B(h \circ \varphi_a) = \rho_\B (h)$ whenever $h \in \B$) so $\| g_a \|_\B = \| f \|_\B$, $\forall a \in \D$. The function $f$ is extremal and then by definition
				\[
					\int_\D |f(\zeta) - f(a)|^p |\varphi_a'(\zeta)|^{\alpha + 2} \: d\mu_\alpha (\zeta) = \int_\D |g_a|^p \: d\mu_\alpha \leq \int_\D |f|^p \: d\mu_\alpha.
				\]
			
			Note that $|\varphi_a'(\zeta)|^2 = \frac{(1 - |a|^2)^2}{|1 - \overline{a}\zeta} = 1 + 4 \Rea \{ \overline{a} \zeta \} + o(|a|)$, and therefore 
				\[
					|\varphi_a'(\zeta)|^{\alpha + 2} = 1 + 2(\alpha + 2) \Rea \{ \overline{a} \zeta \} + o(|a|), \quad a \rightarrow 0.
				\]
			
			Take the partition $\D_a^+ := \{ \zeta \in \D \: : \: |f(\zeta)| > |f(a)| \}$ and $\D_a^- := \D \setminus \D_a^+$.
				
			\begin{itemize}
				\item If $\zeta \in \D_a^-$, using the fact that $p > 1$ and $f(0) = 0$,
					\[
						|f(\zeta) - f(a)|^p \leq 2^p |f(a)|^p = o(|a|), \quad a \rightarrow 0.
					\]
					
					Thus, applying Dominated Convergence Theorem we have
					\begin{eqnarray*}
						\int_{\D_a^-} |f(\zeta) - f(a)|^p |\varphi_a'(\zeta)|^{\alpha + 2} \: d\mu_\alpha (\zeta) & = & o(|a|), \\
						\int_{\D_a^-} |f|^p \: d\mu_\alpha & = & o(|a|),
					\end{eqnarray*}
					for $a \rightarrow 0$.
									
				\item If $\zeta \in \D_a^+$,
					\[
						|f(\zeta) - f(a)|^p = |f(\zeta)|^p - p |f(\zeta)|^{p-2} \Rea \{ f(\zeta) \overline{f'(0)a} \} + o(|a|), \quad a \rightarrow 0.
					\]
					
					Applying again Dominated Convergence Theorem,
					 \begin{eqnarray*}
					 	\int_{\D_a^+} |f - f(a)|^p |\varphi_a'|^{\alpha + 2} \: d\mu_\alpha & = & \int_{\D_a^+} |f|^{p-2} \Rea \left \{ [2(\alpha + 2) |f|^2 \zeta - p f \overline{f'(0)} ] \overline{a} \right \} \: d\mu_\alpha \\
					 		& & + \int_{\D_a^+} |f|^p \: d\mu_\alpha + o(|a|), \quad a \rightarrow 0.
					 \end{eqnarray*}
			\end{itemize}
			
			Summing up, if we write $a = r e^{it}$ with $r > 0$ and $t \in [0, 2\pi)$, we have
				\[
					\Rea \left \{ e^{-it} \int_{\D_a^+} |f(\zeta)|^{p-2} \left [ 2(\alpha + 2) |f(\zeta)|^2 \zeta - p \overline{f'(0)} f(\zeta)\right] \: d\mu_\alpha (\zeta) \right \} \leq \dfrac{o(r)}{r}, \quad r \rightarrow 0^+.
				\]
			
			Using one last time Dominated Convergence Theorem, we prove
				\[
					\int_{\D \setminus \mathcal{Z}(f)} |f(\zeta)|^p \zeta \: d\mu_\alpha (\zeta) = \dfrac{p}{2 (\alpha + 2)} \overline{f'(0)} \int_{\D \setminus \mathcal{Z} (f)} |f(\zeta)|^{p - 2} f(\zeta) \: d\mu_\alpha (\zeta),
				\]
				where $\mathcal{Z}(f)$ is the zero set of $f$. The function $f$ is not identically zero by hypothesis, thus $\mu_\alpha \big ( \mathcal{Z}(f) \big) = 0$ and
				\[
					\int_\D |f(\zeta)|^p \zeta \: d\mu_\alpha (\zeta) = \dfrac{p}{2 (\alpha + 2)} \overline{f'(0)} \int_\D |f(\zeta)|^{p - 2} f(\zeta) \: d\mu_\alpha (\zeta).
				\]
		\end{proof}
		
	\section*{Acknowledgments} 
		
		The author is partially supported by grant PID2019-106870GB-I00 from MICINN (Spain) and by MU Fellowship, reference number FPU/00040.
		
	\nocite{MR1758653, MR2033762, MR2311536}
	\bibliographystyle{}
	\bibliography{Bibliografia}

\end{document}